\documentclass{article}

\usepackage[english]{babel}

\usepackage{amsmath,amsfonts,mathrsfs,amssymb,amsthm,algorithm,algpseudocode}
\usepackage{paralist}
\usepackage{xparse}
\usepackage{authblk}
\usepackage[all]{xypic}
\usepackage[utf8]{inputenc}
\usepackage{lmodern,csquotes}

\title{\bfseries Uniform generation of infinite concurrent runs:\\
the case of trace monoids}
\author[1]{Samy Abbes\thanks{\texttt{samy.abbes@univ-paris-diderot.fr}}}
\author[2]{Vincent Jugé\thanks{\texttt{vincent.juge@u-pem.fr}}}
\affil[1]{Université Paris Diderot/IRIF CNRS UMR 8243}
\affil[2]{Université Paris Est/LIGM CNRS UMR 8049}

\date{Sept. 2017}

\newtheorem{proposition}{Proposition}[section]

\newtheorem{lemma}[proposition]{Lemma}
\newtheorem{theorem}[proposition]{Theorem}
\newtheorem*{proprime}{Proposition~\ref{prop:2}'}

\theoremstyle{definition}
\newtheorem{definition}[proposition]{Definition}

\newcommand{\seq}[2]{(#1_{#2})_{#2\geqslant0}}

\newcommand{\M}{\mathcal{M}}
\newcommand{\tq}{\;:\;}
\newcommand{\unit}{\mathbf{e}}
\newcommand{\R}{\mathcal{R}}
\DeclareDocumentCommand{\size}{ m g }{%
  {\IfNoValueT{#2}{|#1|}\IfNoValueF{#2}{|#1|_{#2}}}%
}
\newcommand{\C}{\mathscr{C}}
\newcommand{\ZZ}{\mathbb{Z}}

\newcommand{\RR}{\mathbb{R}}
\newcommand{\U}{\mathsf{U}}
\newcommand{\G}{G}
\newcommand{\myleq}{\preccurlyeq}
\renewcommand{\Up}[1]{\Uparrow #1}
\newcommand{\up}[1]{\uparrow #1}

\newcommand{\esp}{\mathbb{E}}
\newcommand{\BM}{\partial\M}
\newcommand{\Mbar}{\overline\M}
\newcommand{\Pyr}[1]{\text{\normalfont\sffamily Pyr}_{\Sigma}(a_1)}
\newcommand{\Lk}{\mathscr{L}}
\newcommand{\A}{\mathcal{A}}

\newcommand{\B}{\mathbf{B}}
\newcommand{\iid}{{\normalfont\textsl{i.i.d.}}}

\usepackage{xcolor}

\begin{document}

\maketitle

\begin{abstract}
  We introduce an algorithm for the uniform generation of infinite runs in concurrent systems under a partial order probabilistic semantics. We work with trace monoids as concurrency models. The algorithm outputs on-the-fly approximations of a theoretical infinite run, the latter being distributed according to the exact uniform probability measure. The average size of the approximation grows linearly with the time of execution of the algorithm. The execution of the algorithm only involves distributed computations, provided that some---costly---precomputations have been done.
\end{abstract}

\section{Introduction}
\label{sec:introduction}

Random generation is the theoretical basis for the simulation of systems. For a system that one wishes to simulate, one considers, for a large integer~$n$, the set $E_n$ of its \emph{partial executions}---also called \emph{trajectories}---of length~$n$, which is a finite set in all considered cases. Several techniques have been introduced for the generation of elements of combinatorial structures~\cite{kulkarny90,flajolet94,duchon04:_boltz}. One drawback of these methods, however, is the following: assuming that the random generation has been done for trajectories of some length, it is in most models entirely useless for the random generation of trajectories of a larger size. 
Nevertheless, this simulation task of long executions makes only sense when trying to simulate systems for which the execution time is unbounded. This motivates us to seek direct methods for the random generation of infinite runs of systems.

In this paper, we consider \emph{trace monoids}~\cite{cartier69,diekert95}. A trace monoid is a finitely presented monoid where the relations are only  of the form $ab=ba$ for some pairs $(a,b)$ of generators---trace monoids correspond thus to right-angled Artin-Tits monoids~\cite{davis07}. The minimal generators of the monoid represent the elementary actions of a system; and the imposed commutation relations between generators stand for the concurrency between some elementary actions. The elements of the trace monoid, called \emph{traces}, represent the partial executions of a system featuring concurrency, and the concatenation stands for the ability of chaining several partial executions. 

In order to cope with unbounded executions of a system, elements of the trace monoid are not enough anymore. Fortunately, for trace monoids, there is a well defined space of \emph{infinite traces}---any infinite trace is the limit of a non decreasing sequence of traces. The space of infinite traces is called the \emph{boundary at infinity} of the monoid, as it is reminiscent of the analogous notion in Group theory~\cite{ghys90}. The boundary at infinity is equipped with a probability measure defined as the weak limit of a sequence of finite uniform probability distributions~\cite{abbes15}. This probability measure is called the \emph{uniform measure at infinity} and captures a notion of infinite runs uniformly distributed.

In this paper, we introduce new techniques for the uniform generation of infinite traces. This problem was solved by the first author for some particular trace monoids~\cite{abbes17}. Here we extend these results through new techniques that apply to any trace monoid. What we mean by an algorithm generating infinite traces is an algorithm that runs endlessly, and outputs a random non decreasing sequence of traces~$(\xi_n)_{n\geqslant0}$\,. The sequence $(\xi_n)_{n\geqslant0}$ approximates an infinite trace~$\xi$, which is distributed according to the \emph{exact} uniform measure at infinity. The average size of the output element $\xi_n$ increases linearly with~$n$, and the algorithm only involves distributed computations---which means in particular that the computations can be effectively processed in parallel. 

There are two steps to achieve our goal. First, we study the random generation of elements of the monoid. To this aim, we consider a family of discrete probability distributions which correspond both to the Poincar\'e series and associated sums of Dirac distributions introduced in Group theory, and to the probability distributions used for Boltzmann sampling in Combinatorics. We introduce an iterative method for the random generation, where the iteration runs on the size of the alphabet of generators of the monoid. Secondly, we use a reconstruction result for the uniform probability measure at infinity---this result is proved elsewhere---to derive from the first step a random generation method for infinite traces.

\paragraph*{Challenges.}
\label{sec:challenges}

Some challenges related to the uniform generation of infinite traces, although hidden in the core of the paper, deserve to be made explicit as they cannot be solved in an obvious manner. For each non negative integer~$n$, let $m_n$ be the finite uniform probability distribution on the set $E_n$ of elements of length~$n$ in the monoid. The uniform measure at infinity $\B$ is defined as the weak limit, when $n\to\infty$, of the sequence~$(m_n)_{n\geqslant0}$\,. For this definition to make sense, we have to embed the monoid into a suitable compactification, and this is the role of the boundary at infinity of the monoid. Hence $\B$ is a mathematical object that idealizes what ``a large element of the monoid uniformly distributed'' can be---within the paper, we use an alternative but equivalent definition for~$\B$. Now the point is to notice that the definition of a probability measure as a weak limit has no obvious operational counterpart from the uniform generation point of view. This general difficulty is illustrated in our context as follows: how can a ``boundary'' element be uniformly generated by making choices within a finite horizon? The mere definition of the uniform measure at infinity does not provide the answer. To overcome this difficulty, we need to use a reconstruction result for the uniform measure at infinity. This result provides an alternative description of~$\B$. Thanks to this result, one constructs a random infinite element by concatenating an \iid\ sample of finite, but of unbounded length, random elements of the monoid.

Another difficulty is to obtain, as we do, a random generation procedure which can be performed in a \emph{distributed way}. We give some details on this issue in the last section of the paper.

\paragraph*{Outline.}
\label{sec:outline}

Section~\ref{sec:comb-trace-mono} introduces the essential definitions for trace monoids and recalls combinatorial results. Section~\ref{sec:uniform-measures} introduces probabilistic material, in particular a family of discrete probability distributions attached to each trace monoid and the uniform measure on the space of infinite traces. In Section~\ref{sec:rand-gener-finite}, we introduce a generation algorithm for finite traces. In Section~\ref{sec:rand-gener-infin}, we show how to derive from the results of Section~\ref{sec:rand-gener-finite} an algorithm  which generates uniformly distributed infinite traces. Finally, Section~\ref{sec:distr-impl-algor} discusses the distributed features of the previously introduced algorithms. 

\section{Combinatorics of trace monoids}
\label{sec:comb-trace-mono}

Let $\Sigma$ be a finite set---the alphabet---, and let $\R$ be a reflexive and symmetric binary relation on~$\Sigma$. The \emph{trace monoid} $\M(\Sigma,\R)$ is the presented monoid $\langle\Sigma\;\big|\; ab=ba\text{ for $(a,b)\notin \R$}\rangle$. Its elements are thus equivalence classes of $\Sigma$-words, with respect to the reflexive and transitive closure of the elementary relations $xaby\sim xbay$ for $x,y\in\Sigma^*$ and for $(a,b)\notin \R$. We denote by $\unit$ the neutral element, and we denote by ``$\cdot$'' the concatenation in~$\M(\Sigma,\R)$.

If $\Sigma'$ is a subset of~$\Sigma$, we denote by $\M_{\Sigma'}$ the sub-monoid of $\M(\Sigma,\R)$ generated by~$\Sigma'$. Equivalently, $\M_{\Sigma'}$ is the trace monoid $\M(\Sigma',\R')$, where $\R'=\R\cap(\Sigma'\times\Sigma')$ is the restriction of $\R$ to~$\Sigma'$. In particular $\M(\Sigma,\R)=\M_\Sigma$\,, and our notation amounts to having the relation $\R$ understood.

It is known that elements of $\M_\Sigma$ can be represented by \emph{heaps}~\cite{viennot86}, according to a bijective correspondence which we briefly recall now, following the presentation of~\cite{krattenthaler06}. A \emph{heap}---see Figure~\ref{fig:pqjwdpoqjw},~(i)---is a triple $(P,\myleq,\ell)$, where $(P,\myleq)$ is a poset and $\ell:P\to\Sigma$ is a labeling of $P$ by elements of~$\Sigma$, satisfying the two following properties:
\begin{enumerate}
\item\label{item:1} If $x,y\in P$ are such that $\ell(x) \R\, \ell(y)$, then $x\myleq y$ or $y\myleq x$.
\item The relation $\myleq$ is the transitive closure of the relations from~\ref{item:1}.
\end{enumerate}
More precisely, the heap is the equivalence class of $(P,\myleq,\ell)$, up to isomorphism of labelled partial orders.

Let $x$ be an element of~$\M_\Sigma$\, and let $(P,\myleq,\ell)$ be the corresponding heap. Then the $\Sigma$-words that belong to the equivalence class $x$ are exactly those $\Sigma$-words of the form $\ell(x_1)\dots\ell(x_k)$, where $(x_1,\ldots,x_k)$ is any linearization of the poset $(P,\myleq)$. In particular, for any $a\in\Sigma$, the number of occurrences of $a$ in the words $\ell(x_1)\dots\ell(x_k)$ is constant, we denote it by~$\size x{a}$\,. Similarly, the length of the words in the equivalence class $x$ is constant, we denote it by~$\size x$.

To picture heaps corresponding to traces in~$\M_\Sigma$\,, one represents elements of $\Sigma$ as elementary pieces that can be piled up, one on top of the others. The piling must satisfy the following constraints:
\begin{inparaenum}[(1)]
  \item pieces can only be moved vertically;
  \item pieces labelled by the same letter move along the same vertical lane; and
  \item\label{itensoasaa} two pieces labelled by $a$ and $b$ in $\Sigma$ can be moved independently of each other if and only if $(a,b)\notin \R$. 
\end{inparaenum} Point~\ref{itensoasaa} corresponds to the commutativity relation $a\cdot b=b\cdot a$ which holds in the monoid~$\M_\Sigma$ whenever $(a,b)\notin \R$. See Figure~\ref{fig:pqjwdpoqjw},~(ii). Any heap is obtained by letting elementary pieces fall, from top to bottom.

\begin{figure}[t]
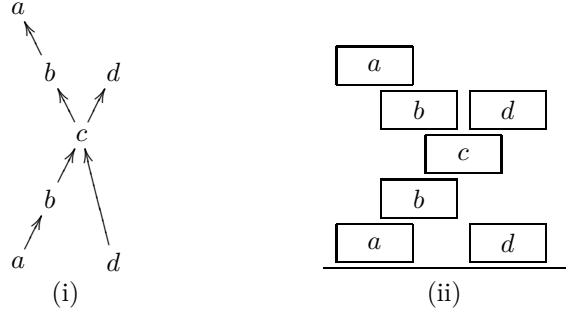

  \centering
\begin{tabular}{ccc}
\xy<.15em,0em>:
0*+{a};(8,16)*+{b}**@{-}?>*\dir{>};
(8,16)*+{\phantom{b}};
(16,32)*+{c}**@{-}?>*\dir{>},
(24,0)*+{d};
(16,32)*+{\phantom{c}}**@{-}?>*\dir{>};
(16,32)*+{\phantom{c}};
(24,48)*+{d}**@{-}?>*\dir{>},
(8,48)*+{b}**@{-}?>*\dir{>};
(8,48)*+{\phantom{b}};
(0,64)*+{a}**@{-}?>*\dir{>},
\endxy
&\strut\hspace{5em}\strut&
\xy
<.12em,0em>:
(0,0)="G",
"G"+(12,6)*{a},
"G";"G"+(24,0)**@{-};"G"+(24,12)**@{-};"G"+(0,12)**@{-};"G"**@{-},
(14,14)="G",
"G"+(12,6)*{b},
"G";"G"+(24,0)**@{-};"G"+(24,12)**@{-};"G"+(0,12)**@{-};"G"**@{-},
(28,28)="G",
"G"+(12,6)*{c},
"G";"G"+(24,0)**@{-};"G"+(24,12)**@{-};"G"+(0,12)**@{-};"G"**@{-},
(42,0)="G",
"G"+(12,6)*{d},
"G";"G"+(24,0)**@{-};"G"+(24,12)**@{-};"G"+(0,12)**@{-};"G"**@{-},
(14,42)="G",
"G"+(12,6)*{b},
"G";"G"+(24,0)**@{-};"G"+(24,12)**@{-};"G"+(0,12)**@{-};"G"**@{-},
(0,56)="G",
"G"+(12,6)*{a},
"G";"G"+(24,0)**@{-};"G"+(24,12)**@{-};"G"+(0,12)**@{-};"G"**@{-},
(42,42)="G",
"G"+(12,6)*{d},
"G";"G"+(24,0)**@{-};"G"+(24,12)**@{-};"G"+(0,12)**@{-};"G"**@{-},
(-4,-2);(72,-2)**@{-}
\endxy\\
(i)&&(ii)
\end{tabular}
\caption{(i)~{\slshape Hasse diagram of the labelled partial order corresponding to the element $x=a\cdot b\cdot d\cdot c\cdot b\cdot a\cdot d$ of the trace monoid $\M(\Sigma,\R)$ with $\Sigma=\{a,b,c,d\}$ and where $\R$ is the reflexive and symmetric closure of $\{(a,b),(b,c),(c,d)\}$.\quad}(ii)~{\slshape Heap of pieces representing the same element.}}
\label{fig:pqjwdpoqjw}
\end{figure}

If a heap is entirely made up of pieces which do not block each other, it is called a \emph{trivial heap} in~\cite{viennot86}. In~$\M_\Sigma$\,, trivial heaps correspond to commutative products $a_1\cdot\ldots\cdot a_k$, where $a_1,\ldots,a_k$ are distinct elements of $\Sigma$ such that $(a_i,a_j)\notin\R$ for all distinct $i$ and~$j$. Hence, trivial heaps correspond bijectively to the cliques of the complementary relation of $\R$ in~$\Sigma$. For brevity, we shall simply call these products \emph{cliques}. We let $\C_\Sigma$ denote the set of cliques of~$\M_\Sigma$\,.

Cliques of $\M_\Sigma$ play an important role for the study of its combinatorics. Associate to the pair $(\Sigma,\R)$ the \emph{M\"obius polynomial} $\mu_\Sigma(X)$ and the generating series $G_\Sigma(X)$  defined as in~\cite{cartier69} by:
\begin{align}
\label{eq:1}
  \mu_\Sigma(X)&=\sum_{\gamma\in\C_\Sigma}(-1)^{\size\gamma}X^{\size\gamma}\,,&
G_\Sigma(X)&=\sum_{x\in\M_\Sigma}X^{\size x}\,.
\end{align}
Then $G_\Sigma(X)$ is a rational series, and the following formula holds~\cite{cartier69,viennot86}:
\begin{gather}
\label{eq:2}
G_\Sigma(X)=\frac1{\mu_\Sigma(X)}\quad\text{in $\RR[[X]]$}.
\end{gather}

More generally, let $\U$ be a subset of~$\Sigma$. For $x\in\M_\Sigma$\,, let us write $\max(x)\subseteq\U$ if the heap $(P,\myleq,\ell)$ corresponding to $x$ has the property that all the maximal elements of the poset $(P,\myleq)$ are labelled by elements in~$\U$. Let $\G_{\Sigma,\U}(X)$ be the generating series of the elements of $\M_\Sigma$ satisfying this property:
\begin{gather*}
\G_{\Sigma,\U}(X)=
  \sum_{\substack{x\in\M_\Sigma\tq\max(x)\subseteq\U}}X^{\size x}\,.  
\end{gather*}

Then the following formula holds~\cite{viennot86}:
\begin{align}
\label{eq:4}
\G_{\Sigma,\U}(X)&=\frac{\mu_{\Sigma\setminus\U}(X)}{\mu_{\Sigma}(X)}\quad\text{in $\RR[[X]]$},
&\text{with\quad}
\mu_{\Sigma\setminus\U}(X)&=\sum_{\gamma\in\C_{\Sigma\setminus\U}}(-1)^{\size\gamma}X^{\size\gamma}\,.
\end{align}
Note that Eq.~\eqref{eq:2} is a particular case of Eq.~\eqref{eq:4}, obtained by taking $\U=\Sigma$ since $G_{\Sigma,\Sigma}(X)=G_\Sigma(X)$ and $\mu_\emptyset(X)=1$.

The following results concerning $\mu_\Sigma(X)$ are found in~\cite{goldwurm00,csikvari13,krob03}.
If $\Sigma\neq\emptyset$, then $\mu_\Sigma(X)$ has a unique root of smallest modulus in the complex plane~. This root, which we denote by~$p_\Sigma$\,, is real positive and is at most~$1$. It coincides with the radius of convergence of the power series~$\G_{\Sigma,\U}(X)$ for any non empty subset $\U$ of~$\Sigma$. Hence:
\begin{gather}
\label{eq:5}
\forall p\in(0,p_\Sigma)\qquad
  \sum_{x\in\M_\Sigma\tq\max(x)\subseteq\U}p^{\size x}=\frac{\mu_{\Sigma\setminus\U}(p)}{\mu_\Sigma(p)}\quad\text{in $\RR$}.
\end{gather}

As a particular case of~\eqref{eq:5}, obtained for $U=\Sigma$, one has:
\begin{gather}
  \label{eq:3}
\forall p\in(0,p_\Sigma)\qquad
G_\Sigma(p)=\frac{1}{\mu_\Sigma(p)}\quad\text{in $\RR$}.
\end{gather}

For all subsets $\U,\U'\subseteq\Sigma$ with $\U\subseteq\U'$, it is clear that $\M_{\Sigma\setminus\U'}\subseteq\M_{\Sigma\setminus\U}$. Hence: $G_{\Sigma\setminus\U'}(p)\leqslant G_{\Sigma\setminus\U}(p)<\infty$ for every $p\in(0,p_{\Sigma\setminus\U})$ and therefore:
\begin{gather}
  \label{eq:19}
\U\subseteq\U'\implies p_{\Sigma\setminus\U}\leqslant p_{\Sigma\setminus\U'}\,.
\end{gather}

Let $\M_0$ be the trace monoid with one generator. Then $G_\M(p)\geqslant G_{\M_0}(p)$ holds for any trace monoid $\M\neq\{\unit\}$ and for every $p\in(0,p_\Sigma)$. Since $\mu_{\M_0}(X)=1-X$, formula~(\ref{eq:3}) yields:
\begin{gather}
\label{eq:13}
  \forall p\in(0,p_\Sigma)\qquad
\mu_\Sigma(p)\leqslant 1-p.
\end{gather}

Finally, the \emph{link} $\Lk(a)$ of an element $a\in\Sigma$ is defined by $\Lk(a)=\bigl\{b\in\Sigma\tq(a,b)\in\R\bigr\}$. Then, for every $a\in\Sigma$, the following identity derives at once from the definition of the M\"obius polynomials:
\begin{gather}
  \label{eq:28}
\mu_\Sigma(X)=\mu_{\Sigma\setminus\{a\}}(X)-X\mu_{\Sigma\setminus\Lk(a)}(X).
\end{gather}

\section{Probability measures on trace monoids}
\label{sec:uniform-measures}

\subsection{Discrete probability distributions}
\label{sec:discr-mult-meas}

Given a trace monoid $\M_\Sigma=\M(\Sigma,\R)$, we introduce a family of discrete probability distributions on~$\M_\Sigma$\,. This family of probability distributions is indexed by a real parameter ranging over the open interval $(0,p_\Sigma)$, where $p_\Sigma$ denotes the root of smallest modulus of the Möbius polynomial of the pair~$(\Sigma,\R)$. For each $p\in(0,p_\Sigma)$, we let $B_{\Sigma,p}$ be the probability distribution on $\M_\Sigma$ defined by:
\begin{gather}
  \label{eq:7}
\forall x\in\M_\Sigma\qquad B_{\Sigma,p}\bigl(\{x\}\bigr)=\mu_\Sigma(p)p^{\size x}\,,\qquad\text{recalling that }
\mu_\Sigma(p)=\frac1{G_\Sigma(p)}\,.
\end{gather}

Let $\leq$ denote the left divisibility relation on~$\M_\Sigma$\,, defined for $x,y\in\M_\Sigma$ by $x\leq y\iff\exists z\in\M_\Sigma\quad x\cdot z = y$. For every $x\in\M_\Sigma$\,, we also define $\Up x=\{y\in\M_\Sigma\tq x\leq y\}$. Then, for every $p\in(0,p_\Sigma)$, $B_{\Sigma,p}$~is the unique probability distribution on $\M_\Sigma$ satisfying the following identities~\cite{abbes15conf}:
\begin{gather}
  \label{eq:8}
\forall x\in\M_\Sigma\qquad B_{\Sigma,p}\bigl(\Up x\bigr)=p^{\size x}\,.
\end{gather}

\begin{definition}
  \label{def:2}
The discrete probability distributions $B_{\Sigma,p}$ for $p\in(0,p_\Sigma)$ are the \emph{multiplicative probabilities} on~$\M_\Sigma$\,.
\end{definition}

Let $\xi$ be a random element in $\M_\Sigma$ distributed according to~$B_{\Sigma,p}$\,. Then the average length of~$\xi$, denoted by $\esp_p\size\xi$, is finite and given by:
\begin{align*}
  \esp_p\size\xi =\mu_\Sigma(p)\sum_{x\in\M_\Sigma}\size x p^{\size x}=\mu_\Sigma(p) pG_\Sigma'(p)=-p\frac{\mu_\Sigma'(p)}{\mu_\Sigma(p)}\,.
\end{align*}

Hence, $\esp_p\size \xi$ is finite for every $p<p_\Sigma$\,, but goes to $\infty$ when $p$ goes to~$p_\Sigma$\,. Informally speaking, the mass of the distribution $B_{\Sigma,p}$ is pushed toward arbitrary large elements of $\M_\Sigma$ when $p\to p_\Sigma$\,. This suggests to introduce infinite traces by a compactification of~$\M_\Sigma$\,. The set of infinite traces will then support the limit of the probability distributions~$B_{\Sigma,p}$\,, when $p\to p_\Sigma$\,. This is formalized next---the remaining of this section can be skipped until the reading of Section~\ref{sec:rand-gener-infin}.

\subsection{Uniform measure at infinity}
\label{sec:unif-meas-bound}

Borrowing material from~\cite{abbes15}, we briefly explain the construction of infinite traces and of the uniform measure on their set.

If $x=\seq xn$ and $y=\seq yn$ are two non decreasing sequences in~$\M_\Sigma$\,, we define $x\sqsubseteq y$ whenever:
\begin{gather}
  \label{eq:18}
\forall n\geqslant0\quad\exists k\geqslant0\quad x_n\leq y_k\,.
\end{gather}
The relation $\sqsubseteq$ is a preorder relation on the set of non decreasing sequences. Let $\asymp$ be the equivalence relation defined by $x\asymp y\iff(x\sqsubseteq y\wedge y\sqsubseteq x)$.  Equivalence classes of non decreasing sequences modulo $\asymp$ are called \emph{generalized traces}, and their set is denoted by~$\Mbar_\Sigma$\,. The set $\Mbar_\Sigma$ is equipped with an ordering relation, denoted by~$\leq$, which is the collapse of the preordering relation~$\sqsubseteq$. 

The partial order $(\M_\Sigma,\leq)$ is embedded into~$(\Mbar_\Sigma,\leq)$, by sending an element $x\in\M_\Sigma$ to the equivalence class of the constant sequence $\seq xn$ with $x_n=x$ for all $n\geqslant0$. Hence we identify $\M_\Sigma$ with its image in~$\Mbar_\Sigma$\,, and we put $\BM_\Sigma=\Mbar_\Sigma\setminus\M_\Sigma$\,. Elements of $\BM_\Sigma$ are called \emph{infinite traces}. 

Visually, infinite traces can be pictured as heaps obtained as in Figure~\ref{fig:pqjwdpoqjw}, but with infinitely many pieces piled up. The relation $x\leq\xi$ for $x\in\M_\Sigma$ and for $\xi\in\BM_\Sigma$ means that the infinite heap $\xi$ can be built by piling up pieces on top of~$x$. Just as the words in the equivalence class of a heap describe the different ways of building the heap by adding one piece after another, the non decreasing sequences in the equivalence class of an infinite heap describe the several ways of building an infinite heap. 

For every $x\in\M_\Sigma$\,, we define the \emph{visual cylinder} of base $x$ as the subset of~$\BM_\Sigma$\,:
\begin{gather}
  \label{eq:20}
\up x=\{\xi\in\BM_\Sigma\tq x\leq\xi\}.
\end{gather}

There is a natural topology on~$\Mbar_\Sigma$ making it a compact metrisable space, inducing the discrete topology on~$\M_\Sigma$\,, and such that every non decreasing sequence $x=\seq xn$ in $\M_\Sigma$ is convergent toward the equivalence class of~$x$. The limit also coincides with the least upper bound in $(\Mbar_\Sigma,\leq)$ of the chain~$\seq xn$\,. Hence:
\begin{gather}
  \label{eq:22}
\text{If $x\in\Mbar_\Sigma$ is the equivalence class of $\seq xn$, then }
x=\bigvee_{n\geqslant0}x_n\,.
\end{gather}
For the topology induced on~$\BM_\Sigma$\,, visual cylinders are both open and closed. The topological space $\BM_\Sigma$ is the \emph{boundary of\/~$\M_\Sigma$}\,.

\emph{Via} the embedding $\M_\Sigma\to\Mbar_\Sigma$\,, the family $(B_{\Sigma,p})_{p\in(0,p_\Sigma)}$ of multiplicative probabilities can now be seen as a family of discrete distributions on the compactification~$\Mbar_\Sigma$ rather than on~$\M_\Sigma$\,. Standard techniques from Functional Analysis allow to prove the weak convergence of~$B_{\Sigma,p}$\,, when $p\to p_\Sigma$\,, toward a probability measure $B_{\Sigma,p_\Sigma}$ on~$\BM_\Sigma$\,, entirely characterized by the following Bernoulli-like identities~\cite{abbes15,abbes15conf}:
\begin{gather}
  \label{eq:9}
\forall x\in\M_\Sigma\qquad B_{\Sigma,p_\Sigma}(\up x)=p_\Sigma^{\size x}\,.
\end{gather}

\begin{definition}
  \label{def:3}
The probability measure $B_{\Sigma,p_\Sigma}$  on\/ $\BM_\Sigma$ is the \emph{uniform measure at infinity}.
\end{definition}

So far, we have thus defined a family of probability measures $B_{\Sigma,p}$ on~$\Mbar_\Sigma$\,, for $p$ ranging over the half-closed interval $(0,p_\Sigma]$. Note the alternative: $B_{\Sigma,p}$~is either concentrated on $\M_\Sigma$ only, in the case where $p<p_\Sigma$\,, or it is concentrated on $\BM_\Sigma$ only, in the case where $p=p_\Sigma$\,. 

 
The remaining of the paper is devoted to the random generation of elements in $\BM_\Sigma$ uniformly distributed; more precisely, we are seeking a random algorithm, the execution of which produces successive finite approximations of elements $\xi\in\BM_\Sigma$ distributed according to~$B_{\Sigma,p_\Sigma}$\,.


For this task, we first focus on the random generation of finite heaps, that is to say, elements of $\M_\Sigma$ distributed according to $B_{\Sigma,p}$ with $p<p_\Sigma$\,. This is the topic of Section~\ref{sec:rand-gener-finite}. We shall then see how to derive random generation techniques for infinite heaps in Section~\ref{sec:rand-gener-infin}.

\section{Random generation of finite traces}
\label{sec:rand-gener-finite}

As a first task, we consider the random generation of elements in a trace monoid $\M_\Sigma=\M(\Sigma,\R)$ distributed according to a probability distribution $B_{\Sigma,p}$ for $p\in(0,p_\Sigma)$. We target an incremental procedure, where elements of $\Sigma$ are added one after the other.

\subsection{Decomposition of traces}
\label{sec:decomposition-traces}

We fix an arbitrary element $a_1\in\Sigma$. Let $x\in\M_\Sigma$ be such that $\size x{a_1}>0$.
Consider a heap $(P,\myleq,\ell)$ representing~$x$. The elements of $P$ labelled by $a_1$ form a non empty chain  according to the axiom~\ref{item:1} of traces recalled at the beginning of Section~\ref{sec:comb-trace-mono}. Let $\alpha$ be the minimum of this chain, and consider $Q=\{\beta\in P\tq\beta\myleq\alpha\}$. Obviously, $Q$~equipped with the restrictions of $\myleq$ and of $\ell$ is a sub-heap of $(P,\myleq,\ell)$, and corresponds thus to an element $y$ in~$\M_\Sigma$\,. We put $V_{a_1}(x)=y$, and we define (see Figure~\ref{fig:oiqoqihwq}):
\begin{gather}
  \label{eq:10}
\Pyr{a_1}=\bigl\{V_{a_1}(x)\tq x\in\M_\Sigma\,,\quad\size x{a_1}>0\bigr\}.
\end{gather}

\begin{figure}[t]
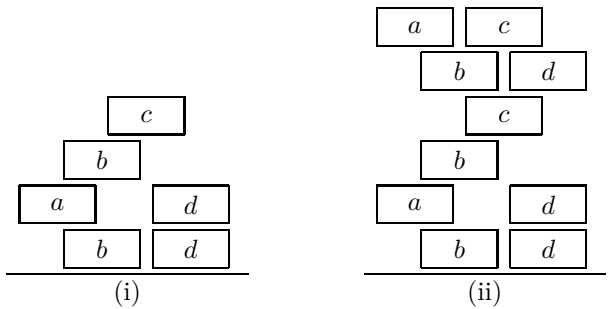

  \centering
\begin{tabular}{ccc}
  \xy
<.12em,0em>:
(14,0)="G",
"G"+(12,6)*{b},
"G";"G"+(24,0)**@{-};"G"+(24,12)**@{-};"G"+(0,12)**@{-};"G"**@{-},
(0,14)="G",
"G"+(12,6)*{a},
"G";"G"+(24,0)**@{-};"G"+(24,12)**@{-};"G"+(0,12)**@{-};"G"**@{-},
(14,28)="G",
"G"+(12,6)*{b},
"G";"G"+(24,0)**@{-};"G"+(24,12)**@{-};"G"+(0,12)**@{-};"G"**@{-},
(28,42)="G",
"G"+(12,6)*{c},
"G";"G"+(24,0)**@{-};"G"+(24,12)**@{-};"G"+(0,12)**@{-};"G"**@{-},
(42,0)="G",
"G"+(12,6)*{d},
"G";"G"+(24,0)**@{-};"G"+(24,12)**@{-};"G"+(0,12)**@{-};"G"**@{-},
(42,14)="G",
"G"+(12,6)*{d},
"G";"G"+(24,0)**@{-};"G"+(24,12)**@{-};"G"+(0,12)**@{-};"G"**@{-},
(-4,-2);(72,-2)**@{-}
\endxy
&&  \xy
<.12em,0em>:
(14,0)="G",
"G"+(12,6)*{b},
"G";"G"+(24,0)**@{-};"G"+(24,12)**@{-};"G"+(0,12)**@{-};"G"**@{-},
(0,14)="G",
"G"+(12,6)*{a},
"G";"G"+(24,0)**@{-};"G"+(24,12)**@{-};"G"+(0,12)**@{-};"G"**@{-},
(14,28)="G",
"G"+(12,6)*{b},
"G";"G"+(24,0)**@{-};"G"+(24,12)**@{-};"G"+(0,12)**@{-};"G"**@{-},
(28,42)="G",
"G"+(12,6)*{c},
"G";"G"+(24,0)**@{-};"G"+(24,12)**@{-};"G"+(0,12)**@{-};"G"**@{-},
(42,0)="G",
"G"+(12,6)*{d},
"G";"G"+(24,0)**@{-};"G"+(24,12)**@{-};"G"+(0,12)**@{-};"G"**@{-},
(42,14)="G",
"G"+(12,6)*{d},
"G";"G"+(24,0)**@{-};"G"+(24,12)**@{-};"G"+(0,12)**@{-};"G"**@{-},
(14,56)="G",
"G"+(12,6)*{b},
"G";"G"+(24,0)**@{-};"G"+(24,12)**@{-};"G"+(0,12)**@{-};"G"**@{-},
(0,70)="G",
"G"+(12,6)*{a},
"G";"G"+(24,0)**@{-};"G"+(24,12)**@{-};"G"+(0,12)**@{-};"G"**@{-},
(42,56)="G",
"G"+(12,6)*{d},
"G";"G"+(24,0)**@{-};"G"+(24,12)**@{-};"G"+(0,12)**@{-};"G"**@{-},
(28,70)="G",
"G"+(12,6)*{c},
"G";"G"+(24,0)**@{-};"G"+(24,12)**@{-};"G"+(0,12)**@{-};"G"**@{-},
(-4,-2);(72,-2)**@{-}
\endxy
\\
(i)&\strut\qquad\strut&(ii)
\end{tabular}
\caption{(i)~\textsl{Heap representing the $c$-pyramidal element $b\cdot a\cdot b\cdot d\cdot d\cdot c$ in the trace monoid~$\M_\Sigma$\,, where $(\Sigma,\R)$ is as in Figure~\ref{fig:pqjwdpoqjw}.}\quad (ii)~\textsl{An element $x\in\M_\Sigma$ such that $V_c(x)$ corresponds to the $c$-pyramidal heap represented on the left. For this element~$x$, we also have $V_a(x)=b\cdot a$, $V_b(x)=b$ and $V_d(x)=d$.}}
  \label{fig:oiqoqihwq}
\end{figure}

\begin{definition}
  \label{def:1}
Elements of $\Pyr{a_1}$ are said to be $a_1$-pyramidal in~$\M_\Sigma$\,.
\end{definition}

Any element  $u\in\Pyr{a_1}$ is a product
$u=z\cdot a_1$\,, where $\size z{a_1}=0$, and the element labelled by $a_1$ in the heap representing $u$ is the unique maximal element of this heap. Equivalently, $z$~belongs to the sub-monoid $\M_{\Sigma\setminus\{a_1\}}$ and all the maximal elements of the heap representing $z$ belong to the link of~$a_1$, $\Lk(a_1)=\{b\in\Sigma\tq (a_1,b)\in\R\}$. Conversely, any element $u=z\cdot a_1$ of this form belongs to~$\Pyr{a_1}$, so we have:
\begin{gather}
  \label{eq:11}
\Pyr{a_1}=\bigl\{z\cdot a_1\tq z\in\M_{\Sigma\setminus\{a_1\}}\,,\quad\max(z)\subseteq\Lk(a_1)\bigr\}.
\end{gather}

The successive occurrences of $a_1$ within an element of $\M_\Sigma$ are associated with $a_1$-pyramidal elements, which yields the following decomposition result. The result, visually intuitive, is elementary and belongs to trace theory, hence we omit its proof; see an example below.

\begin{proposition}
  \label{prop:3}
Let $\M_\Sigma$ be a trace monoid, let $a_1\in\Sigma$, and let $x\in\M_\Sigma$\,. Then there exists a unique integer $k\geq0$, given by $k=\size x{a_1}$\,, and a unique tuple $(u_0,\ldots,u_{k-1},u_k)$ such that:
\begin{align}
  \label{eq:12}
u_0,\dots,u_{k-1}&\in\Pyr{a_1}\,,&
\text{and\quad }u_k&\in\M_{\Sigma\setminus\{a_1\}}\,,&
\text{and\quad }x&=u_0\cdot\ldots\cdot u_k\,.
\end{align}
Furthermore, for any subset $T\subseteq\Sigma$ such that $a_1\in T$, we have:
\begin{gather}
  \label{eq:14}
\max(x)\subseteq T\iff\max(u_k)\subseteq T.
\end{gather}
\end{proposition}

Example: for the element $x=b\cdot a\cdot b\cdot d\cdot d\cdot c\cdot b\cdot d\cdot a\cdot c$ represented in Figure~\ref{fig:oiqoqihwq}, (ii), the decomposition with respect to the successive occurrences of $a_1=c$ is the following:
\begin{align*}
  k&=2,&
u_0&=b\cdot a\cdot b\cdot d\cdot d\cdot c,&
u_1&=b\cdot d\cdot c,&u_2&=a.
\end{align*}

\subsection{Probabilistic analysis}
\label{sec:prob-analys}

We study the probabilistic counterpart of the decomposition stated in Proposition~\ref{prop:3}. Throughout Section~\ref{sec:prob-analys}, we fix a trace monoid $\M_\Sigma$ with $\Sigma\neq\emptyset$ and an element $a_1\in\Sigma$.

The combinatorial identities of Section~\ref{sec:comb-trace-mono} yield an expression, which will be useful later, for the probability of occurrence of the given element $a_1\in\Sigma$.

\begin{lemma}
  \label{prop:1}
Let\/ $\M_\Sigma$ be a trace monoid, let $p\in(0,p_\Sigma)$, and let $\xi$ be a random element in $\M_\Sigma$ distributed according to~$B_{\Sigma,p}$\,. Let $a_1\in\Sigma$, and let $r=B_{\Sigma,p}\bigl(\size{\xi}{a_1}>0\bigr)$. Then: 
\begin{gather}
\label{eq:16}
r=1-\frac{\mu_\Sigma(p)}{\mu_{\Sigma\setminus\{a_1\}}(p)}=p\frac{\mu_{\Sigma\setminus\Lk(a_1)}(p)}{\mu_{\Sigma\setminus\{a_1\}}(p)}\,.
\end{gather}
\end{lemma}

\begin{proof}
Let us prove the left equality in~\eqref{eq:16}.
We have $1-r=B_{\Sigma,p}\bigl(\size{\xi}{a_1}=0\bigr)=B_{\Sigma,p}\bigl(\xi\in\M_{\Sigma\setminus\{a_1\}}\bigr)$. Hence the probability $1-r$ evaluates as follows, taking into account~\eqref{eq:7} on the one hand, and applying the formula~\eqref{eq:3} with the trace monoid $\M_{\Sigma\setminus\{a_1\}}$ on the other hand:
\begin{align*}
  1-r=\mu_\Sigma(p)\cdot\Bigl(\sum_{x\in\M_{\Sigma\setminus\{a_1\}}}p^{\size x}\Bigr)=\frac{\mu_\Sigma(p)}{\mu_{\Sigma\setminus\{a_1\}}(p)}\,,
\end{align*}
which is the left equality in~\eqref{eq:16}. Note that the series is necessarily convergent, hence the equality~\eqref{eq:3} applies. The equality with the right member in~\eqref{eq:16} derives directly from the identity~\eqref{eq:28}.
\end{proof}

Consider a probability distribution on~$\M_\Sigma$ and a random element $\xi\in\M_\Sigma$\,. Then $k$ and the elements $u_0,\dots,u_k$ associated to $\xi$ as in Proposition~\ref{prop:3} can be seen as random variables. To underline this point of view, we introduce the notations $K(\xi)$, and $U_0(\xi),\ldots,U_K(\xi)$ such that:
\begin{align}
\label{eq:17}
  \xi&=U_0\cdot\ldots\cdot U_K
\end{align}
with $K=\size{\xi}{a_1}$\,, $U_0,\ldots,U_{K-1}\in\Pyr{a_1}$ and $U_K\in\M_{\Sigma\setminus\{a_1\}}$\,.

It is straightforward to explicitly describe the laws of these random variables: this is stated in Proposition~\ref{prop:2}. In turn, their description
provides the needed hint to formulate a \emph{reconstruction result} for the probability distribution~$B_{\Sigma,p}$\,, which we state in Proposition~\ref{prop:2}'. For random generation purposes, it is the latter result in which we are most interested. The two propositions~\ref{prop:2} and~\ref{prop:2}' are essentially equivalent, they mainly differ in their formulation. Hence we prove both of them simultaneously, based on the following lemma.

\begin{lemma}
\label{lem:3}
Let $L_0,L_1,\ldots,L_k$ be subsets of $\M_\Sigma$, and let $\varphi : L_0 \times \ldots \times L_k \to \M_\Sigma$ be defined by
$\varphi(\ell_0,\ldots,\ell_k) = \ell_0 \cdot \ldots \cdot \ell_k$.
Assume that $\varphi$ is injective, and let $H$ be the set $\varphi(L_0 \times \ldots \times L_k)$.
Then, for all $p \in (0,p_\Sigma)$, we have $B_{\Sigma,p}(H) = \mu_\Sigma(p)^{-k} \prod_{i=0}^k B_{\Sigma,p}(L_i)$.
%
\end{lemma}

\begin{proof}
According to~\eqref{eq:7}, we have $B_{\Sigma,p}(H) = \mu_\Sigma(p) \sum_{x \in H}p^{\size{x}}$\,. Therefore, using that $\varphi$ is injective and that the length function $\size{\cdot}$ is additive:
\begin{align*}
  B_{\Sigma,p}(H)&=\mu_\Sigma(p)\cdot\Bigl(\sum_{(\ell_0,\dots,\ell_k)\in L_0\times\dots\times L_k}
p^{\size{\ell_0}}\cdot\ldots\cdot p^{\size{\ell_k}}\Bigr)=\mu_\Sigma(p)\cdot\prod_{i=0}^k\Bigl(\sum_{\ell\in L_i}p^{\size{\ell}}\Bigr)\,.
\end{align*}
Taking into account that $B_{\Sigma,p}(L_i)=\mu_\Sigma(p)\cdot\Bigl(\sum_{\ell\in L_i}p^{\size{\ell}}\Bigr)$, which is valid for every integer $i\in\{0,\dots,k\}$, the result follows.
\end{proof}

\begin{proposition}
  \label{prop:2}
Let $T$ be a non-empty subset of\/ $\Sigma$ and let $a_1$ be an element of~$T$.
Assume that $\xi$ is a random element in a trace monoid\/ $\M_\Sigma$ distributed according to the probability distribution $B_{\Sigma,p}$
with $p\in(0,p_\Sigma)$ conditionally on\/ $\{\max(\xi)\subseteq T\}$.
Let $K\in\ZZ_{\geqslant0}$, $U_0,\ldots,U_{K-1}\in\Pyr{a_1}$ and $U_K\in\M_{\Sigma\setminus\{a_1\}}$ be
the random variables defined as in\/~\eqref{eq:17} with respect to~$a_1$.
Finally, let $r=B_{\Sigma,p}\bigl(\size{\xi}{a_1}>0\bigr)$. Then:
\begin{enumerate}
\item\label{item:2} The integer $K$ follows a geometric law of parameter~$r$:
  \begin{align}
\label{eq:15}\forall k\in\ZZ_{\geqslant0}\quad B_{\Sigma,p}\bigl(K=k \mid \max(\xi) \subseteq T\bigr) & = (1-r)r^k.
  \end{align}
  
 \item\label{item:3} 
For every non negative integer~$k$, and conditionally on\/ $\{K=k\}$:
\begin{enumerate}
\item\label{item:3a} The variables $U_0,\ldots,U_{k-1},U_k$ are independent.
\item\label{item:3b} The variables $U_0,\dots,U_{k-1}$ are identically distributed. For each integer $i\in\{0,\dots,k-1\}$,
let $V_i$ be the unique element in $\M_{\Sigma\setminus\{a_1\}}$ such that $U_i=V_i\cdot a_1$\,.
The common distribution of\/ $V_0,\dots,V_{k-1}$ is that of an element $\xi\in\M_{\Sigma\setminus\{a_1\}}$
distributed according to $B_{\Sigma\setminus\{a_1\},p}$ conditionally on\/ $\{\max(\xi)\subseteq\Lk(a_1)\}$.
\item\label{item:3c} The variable $U_k$ is distributed in $\M_{\Sigma\setminus\{a_1\}}$ according to the probability distribution $B_{\Sigma\setminus\{a_1\},p}$
conditionally on\/ $\{\max(\,\cdot\,)\subseteq T\}$.
\end{enumerate}
\end{enumerate}
\end{proposition}

The operational variant of Proposition~\ref{prop:2}, in the form of a reconstruction result, is the following.

\begin{proprime}
Let\/ $\M_\Sigma$ be a trace monoid, let $T$ be a non empty subset of\/~$\Sigma$,
let $a_1\in T$, let $p\in(0,p_\Sigma)$ and let $r=1-\mu_\Sigma(p)/\mu_{\Sigma\setminus\{a_1\}}(p)$. 

Then $p<p_{\Sigma\setminus\{a_1\}}$, and therefore $B_{\Sigma\setminus\{a_1\},p}$ is a well defined probability distribution on~$\M_{\Sigma\setminus\{a_1\}}$\,. Furthermore, $r\in(0,1)$.

Consider the following random variables: $K$~is an integer random variable with geometric law of parameter~$r$; and for every non negative integer~$k$,
conditionally on\/ $\{K=k\}$, the variables $V_0,\dots,V_{k-1},U$ are independent and with values in~$\M_{\Sigma\setminus\{a_1\}}$;
the law of\/ $U$ is~$B_{\Sigma\setminus\{a_1\},p}$ conditionally on $\{\max(U) \subseteq T\}$
and $V_0,\dots,V_{k-1}$ are identically distributed, with law $B_{\Sigma\setminus\{a_1\},p}$ conditionally on $\{\max(V_i)\subseteq\Lk(a_1)\}$. Let finally:
\begin{gather}
\label{eq:21}
\xi=(V_0\cdot a_1)\cdot\ldots\cdot( V_{K-1}\cdot a_1)\cdot U\,.
\end{gather}
Then $\xi$ is distributed according to~$B_{\Sigma,p}$\,conditionally on $\{\max(\xi) \subseteq T\}$.
\end{proprime}

Observe that the real $r$ defined in the statement of Proposition~\ref{prop:2}' is indeed the same real than the one defined in Proposition~\ref{prop:2}, thanks to Lemma~\ref{prop:1}.

\begin{proof}
Since Proposition~\ref{prop:2}' is a corollary of Proposition~\ref{prop:2},
we focus on proving the latter result only.

Let $k$ be a non negative integer, let $L_0=\dots=L_{k-1}=\Pyr{a_1}$, $L_k=\{x\in\M_{\Sigma\setminus\{a_1\}}\tq\max(x)\subseteq T\}$, and $L=L_0\times\dots\times L_{k}$\,. Proposition~\ref{prop:3} implies that the mapping $\varphi:L\to\M$ defined by $\varphi(\ell_0,\dots,\ell_k)=\ell_0\cdot\ldots\cdot \ell_k$ is a bijection from $L$ onto the subset $\U_k$ of $\M_\Sigma$ defined by:
\begin{gather*}
\U_k=\bigl\{\xi\in\M_\Sigma\tq  K(\xi)= k\text{ and }\max(\xi)\subseteq T\bigr\}.
\end{gather*}

By the definition~\eqref{eq:7} of~$B_{\Sigma,p}$\,, the characterization~\eqref{eq:11} of~$\Pyr{a_1}$, the summation formula~\eqref{eq:5},  and the value for $r=B_{\Sigma,p}\bigl(\size{\xi}{a_1}>0\bigr)$ found in Lemma~\ref{prop:1}, we have for every integer $i\in\{0,\dots,k-1\}$:
\begin{align*}
B_{\Sigma,p}(L_i)&=\mu_\Sigma(p)\cdot\Bigl(\sum_{{z\in\M_{\Sigma\setminus\{a_1\}}\tq
\max(z)\subseteq\Lk(a_1)}}p^{\size{z\cdot a_1}}\Bigr)\\
&=\mu_\Sigma(p)\frac{p\mu_{\Sigma\setminus\Lk(a_1)}(p)}{\mu_{\Sigma\setminus\{a_1\}}(p)}=r\mu_\Sigma(p)\,,
\\
B_{\Sigma,p}(L_k)&=\mu_\Sigma(p)\cdot\frac{\mu_{\Sigma\setminus T}(p)}{\mu_{\Sigma\setminus\{a_1\}}(p)}\,.
\end{align*}

Therefore, applying Lemma~\ref{lem:3}:
\begin{align}
\label{eq:27}
  B_{\Sigma,p}(\U_k)&=\mu_\Sigma(p)^{-k}\prod_{i=0}^kB_{\Sigma,p}(L_i)=\mu_\Sigma(p)\frac{\mu_{\Sigma\setminus T}(p)}{\mu_{\Sigma\setminus\{a_1\}}(p)}r^k\,.
\end{align}
Since $B_{\Sigma,p}\bigl(K=k\;\big|\;\max(\xi)\subseteq T\bigr)$ is proportional to~$B_{\Sigma,p}(\U_k)$, the result of point~\ref{item:2} in Proposition~\ref{prop:2} follows.

For proving point~\ref{item:3} of Proposition~\ref{prop:2}, we  fix an integer $k\geq0$.  On $\bigl\{\max(\xi)\subseteq T\bigr\}\cap\{K=k\}$, the variable $U_K$ only takes values $u_k\in\M_{\Sigma\setminus\{a_1\}}$ such that $\max(u_k)\subseteq T$ according to Proposition~\ref{prop:3}.   The latter proposition also implies, for any $u_0,\dots,u_{k-1}\in\Pyr{a_1}$ and $u_k\in\M_{\Sigma\setminus\{a_1\}}$ such that $\max(u_k)\subseteq T$\,:
\begin{align*}
B_{\Sigma,p}(U_0=u_0,\dots,U_K=u_k\;\big|\;\U_k\bigr)&=\Bigl(\frac{\mu_\Sigma(p)}{B_{\Sigma,p}(\U_k)}\Bigr)\cdot p^{\size{u_0}}\cdot\ldots\cdot p^{\size{u_k}}\,.
\end{align*}

The above conditional law has a product form, which shows that the $U_i$ are independent. It also shows that the laws of $U_i$ for $i\in\{0,\dots,k-1\}$, which all have the same support, are all proportional to~$p^{\size{u_i}}$, hence $U_0,\dots,U_{k-1}$ are identically distributed. It follows that the random variables $V_0,\dots,V_{k-1}$ are also independent and identically distributed, with law $B_{\Sigma,p}\bigl(V_i=v_i\;|\;\U_k\bigr)$ proportional to~$p^{\size{v_i}}$. Since $B_{\Sigma\setminus\{a_1\},p}\bigl(v_i\;\big|\;\max(v_i)\subseteq\Lk(a_1)\bigr)$ has the same support as $V_i$ and is also proportional to~$p^{\size{v_i}}$, both laws coincide. An analogous argument applies to~$U_k$\,, completing the proof of point~\ref{item:3} of Proposition~\ref{prop:2}.
\end{proof}

\subsection{Random generation of traces}
\label{sec:rand-gener-trac}

We can now use the statement of Proposition~\ref{prop:2}' to start building random algorithms. The first idea is to produce random elements of $\M_\Sigma$ distributed according
to a target distribution~$B_{\Sigma,p}$ conditionally on constraints of the form $\{\xi \in \M_\Sigma \tq \max(\xi) \subseteq T\}$,
provided that we can use a random algorithm able to produce elements of $\M_{\Sigma\setminus\{a_1\}}$ distributed according to~$B_{\Sigma\setminus\{a_1\},p}$
conditionally on similar constraints.

\begin{lemma}
  \label{cor:2}
Let\/ $\M_\Sigma$ be a trace monoid, let $T$ be a non empty subset of\/~$\Sigma$, let $a_1\in T$, and let $p\in(0,p_\Sigma)$.
Assume given a random algorithm $\A$ that, when given as input a subset $X$ of\/~$\Sigma$,
outputs an element $\xi\in\M_{\Sigma\setminus\{a_1\}}$ distributed
according to the probability distribution $B_{\Sigma\setminus\{a_1\},p}(\cdot \mid \max(\xi) \subseteq X)$\,.

Then Algorithm~\ref{algowfdoihjw} described below in pseudo-code outputs with probability~$1$ an element $\xi\in\M_\Sigma$
distributed according to the probability distribution~$B_{\Sigma,p}$ conditionally on $\{\max(\xi) \subseteq T\}$\,.

\begin{algorithm}
\caption{\quad Outputs $\xi\in\M_\Sigma$ distributed according to $B_{\Sigma,p}( \cdot \mid \max(\xi) \subseteq T)$}
\label{algowfdoihjw}
\begin{algorithmic}[1]
\Require{---}\Comment{No input}
\State
$r\gets{1-\mu_\Sigma(p)/\mu_{\Sigma\setminus\{a_1\}}(p)}$\Comment{Needed parameter}
\State
$K\gets{\mathcal{G}(r)}$\Comment{Random integer with a geometric law}
\State
$\xi\gets\unit$\Comment{Initialization}
\For{$i=0$ to $K-1$}
\State
$v\gets\text{output of $\A(\Lk(a_1))$}$
\State
$\xi\gets{\xi\cdot v\cdot a_1}$\Comment{$v$ corresponds to $V_i$ in~\eqref{eq:21}}
\EndFor
\State $u\gets\text{output of $\A(T)$}$\Comment{$u$ corresponds to $U$ in~\eqref{eq:21}}
\State
$\xi\gets \xi\cdot u$
\State
\Return{$\xi$}
\end{algorithmic}
\end{algorithm}
\end{lemma}

\begin{proof}
%
The fact that Algorithm~\ref{algowfdoihjw} outputs an element $\xi\in\M_\Sigma$ distributed according to
$B_{\Sigma,p}(\cdot \mid \max(\xi) \subseteq T)$ is an immediate consequence of Proposition~\ref{prop:2}'.
%
%
\end{proof}

We are now ready to produce a recursive algorithm that outputs random elements in $\M_\Sigma$ distributed according to $B_{\Sigma,p}$ for $p\in(0,p_\Sigma)$.

\begin{theorem}
  \label{thr:1}
Let\/ $\M_\Sigma$ be a trace monoid and let $p\in(0,p_\Sigma)$.
Assume that, for all subsets $X$ of\/~$\Sigma$, the real $\mu_X(p)$ has been precomputed.
Algorithm~\ref{algowfdoihjw2}, described in pseudo-code below, and provided with
an input\/~$(\Sigma, \Sigma)$, outputs an element $\xi \in \M_S$ distributed according to~$B_{\Sigma,p}$\,.

We also assume that every function call, variable assignation and multiplication in $\M_\Sigma$
takes a constant number of steps, and that the call to a routine outputting a random integer $X$ takes a number of steps bounded by~$X$.
Then Algorithm~\ref{algowfdoihjw2}, when outputting an element $\xi$ of\/~$\M_\Sigma$\,,
requires the execution of $\mathcal{O}(\size{\Sigma} (\size{\xi}+1))$ steps.

%

\begin{algorithm}
  \caption{\quad Outputs $\xi\in\M_S$ distributed according to $B_{S,p}( \cdot \mid \max(\xi) \subseteq T)$}
\label{algowfdoihjw2}
  \begin{algorithmic}[1]
\Require{Subsets $S$ and $T$ of $\Sigma$}
\If{$S \cap T = \emptyset$}
\State
\Return{$\unit$}
\Else
\State
\textbf{choose} $a_1 \in S \cap T$
\State
$r\gets{1-\mu_S(p)/\mu_{S\setminus\{a_1\}}(p)}$
\State
$K\gets{\mathcal{G}(r)}$\Comment{Random integer with a geometric law}
\State
$\xi\gets\unit$
\For{$i=0$ to $K-1$}
\State
$v\gets\text{output of Algorithm~\ref{algowfdoihjw2} on input $(S \setminus\{a_1\},\Lk(a_1))$}$
\State
$\xi\gets{\xi\cdot v\cdot a_1}$
\EndFor
\State
$u\gets\text{output of Algorithm~\ref{algowfdoihjw2} on input $(S \setminus\{a_1\},T)$}$
\State
$\xi\gets \xi\cdot u$
\State
\Return{$\xi$}
\EndIf
  \end{algorithmic}
\end{algorithm}
\end{theorem}

\begin{proof}
We prove by induction on $S$ (for the inclusion ordering) that Algorithm~\ref{algowfdoihjw2},
on input $(S,T)$, outputs a random element $\xi_i\in\M_S$ distributed according to $B_{S,p}(\cdot \mid \max(\xi) \subseteq T)$.

The result is trivial for $S = \emptyset$ and for $S \cap T = \emptyset$. Assuming that it holds for all strict subsets of~$S$,
let us apply Lemma~\ref{cor:2} with $S$ in place of $\Sigma$ and with Algorithm~\ref{algowfdoihjw2} on input $(S\setminus\{a_1\},X)$ in place of $\A$ on input~$X$.
Doing so, we prove that $\A(S,T)$ outputs a random element $\xi\in\M_S$ distributed according to~$B_{S,p}\bigl(\cdot \mid \max(\xi) \in T\bigr)$\,. 

The distributions $B_{\Sigma,p}$ and $B_{\Sigma,p}\bigl(\cdot \mid \max(\xi) \subseteq \Sigma\bigr)$
coincide, which completes the proof.

Furthermore, let $f(n,k)$ be the maximal number of steps executed by Algorithm~\ref{algowfdoihjw2} on input $(S,T)$ when outputting an element $\xi$,
where $n = \size{S}$ and $k = \size{\xi}$. By assumption, executing Algorithm~\ref{algowfdoihjw2} on input $(S,T)$ drawing an integer~$K$,
requires at most $\kappa (1 + K) + \sum_{i=0}^K f(n,\ell_i)$ steps, where $\kappa$ is a constant and where $\sum_{i=0}^K \ell_i = k - K$.
Hence, we prove by induction on $n$ that $f(n,k) \leqslant \kappa (n+1)(k+1)$.

Indeed, for $n = 0$, we have $f(n,k) \leqslant \kappa$ independently of $k$,
which proves the complexity bound of Theorem~\ref{thr:1} for $\Sigma = \emptyset$.
Then, provided that our induction hypothesis holds for $n-1$, then $K \leqslant k$, and therefore
\begin{gather*}
\begin{aligned}
f(n,k) & \leqslant \kappa(1+K) + \sum_{i=0}^K \kappa n(\ell_i+1) = \kappa(1+K) + \kappa n(k - K) + \kappa n (K+1) \\
& \leqslant \kappa n(k+1) + \kappa(1+K) \leqslant \kappa (n+1)(k+1),
\end{aligned}
\end{gather*}
which proves our induction hypothesis for $n$, thereby completing the proof of Theorem~\ref{thr:1}.
\end{proof}

\section{Random generation of infinite traces}
\label{sec:rand-gener-infin}

Let $\M_\Sigma=\M(\Sigma,\R)$ be a trace monoid and let $\BM_\Sigma$ be its boundary (see Section~\ref{sec:unif-meas-bound}). In the remaining of this section, we assume that $\M_\Sigma$ is \emph{irreducible}, meaning that the graph $(\Sigma,\R)$ is connected. Indeed, if it is not the case, then the different connected components $\Sigma_1,\dots,\Sigma_k$ induce submonoids $\M_{\Sigma_1},\dots,\M_{\Sigma_k}$\,. Then, by construction of~$\M_\Sigma$\,, one has $a\cdot b=b\cdot a$ for any pair $(a,b)\in\Sigma_i\times\Sigma_j$ with $i\neq j$. It follows that $\M_\Sigma$ is isomorphic to the direct product $\M_{\Sigma_1}\times\dots\times\M_{\Sigma_k}$\,. In other words, the irreducible components of $\M_\Sigma$ do not interact with each other, which justifies the reduction to the irreducible case---see~\cite{abbes15conf} for a precise description of the uniform measure on $\BM_\Sigma$ if $\M_\Sigma$ is not irreducible.

Hence, we fix an irreducible trace monoid~$\M_\Sigma$\,. To shorten the notations, we denote by $\B$ the uniform measure on~$\BM_\Sigma$\,. Recall from Section~\ref{sec:unif-meas-bound} that $\B$ is characterized by its values on visual cylinders: $\B\bigl(\up x\bigr)=p_\Sigma^{\size x}$ for all $x\in\M_\Sigma$\,, where $\up x=\{\xi\in\BM_\Sigma\tq x\leq\xi\}$.

For the uniform generation of infinite traces, we need a reconstruction result analogous to Proposition~\ref{prop:2}', but that would apply to~$\B$. For this, pick $a_1\in\Sigma$. We extend the partial mapping $V_{a_1}:\M_\Sigma\to\Pyr{a_1}$ introduced in Section~\ref{sec:rand-gener-finite}. Let $\xi\in\BM_\Sigma$\,, and let $\seq xn$ be a non decreasing sequence in $\M_\Sigma$ such that $\xi=\bigvee_{n\geqslant0}x_n$\,. We symbolically write $a_1\in\xi$ if $\sup_{n\geqslant0}\size{x_n}{a_1}>0$, and this property does not depend on the chosen sequence~$\seq xn$. If $a_1\in\xi$, then the sequence $\bigl(V_{a_1}(x_n)\bigr)_{n\geqslant q}$ is well defined for $q$ large enough, and it is constant. We define:
\begin{gather}
\label{eq:23}
V_{a_1}(\xi)=
V_{a_1}(x_q)\,,
\end{gather}
which is an $a_1$-pyramidal element of $\M_\Sigma$ independent of the chosen sequence $\seq xn$ and of the integer $q$ chosen large enough.

Since $\M_\Sigma$ is assumed to be irreducible, we know by \cite[Prop.~5.6]{abbes16} that $V_{a_1}$ is defined on $\BM_\Sigma$ with $\B$-probability~$1$, and the law of $V_{a_1}$ is given by:
\begin{gather}
  \label{eq:24}
\forall v\in\Pyr{a_1}\qquad\B\bigl(V_{a_1}=v\bigr)=p_\Sigma^{\size  v}\,.
\end{gather}

Furthermore, by \cite[Th.~6.1]{abbes17}, we have the following reconstruction result for~$\B$: \emph{let\/ $(V_n)_{n\geqslant1}$ be
an \iid\ sequence of random variables in $\Pyr{a_1}$ with law $R(v)=p_\Sigma^{\size v}$ for $v\in\Pyr{a_1}$\,, and
let $\xi=\bigvee_{n\geqslant1}(V_1\cdot\ldots\cdot V_n)$\,; then $\xi$ is an infinite trace  distributed according to~$\B$.}

This reconstruction result tells us that if we can simulate the law of~$V_{a_1}$\,, then we can simulate the law $\B$ by simply concatenating an \iid\ sample of~$V_{a_1}$\,.
Our last task is thus to simulate the law~\eqref{eq:24}, and for this we Lemma~\ref{lem:2} below.

It must be noted that a more naive approach, consisting for example in concatenating \iid\ samples of elements in $\M_\Sigma$ distributed according to some distribution $B_{\Sigma,p}$ with $p<p_\Sigma$\,, would not work in general: see a counter-example in~\cite[\S~6.1.2]{abbes17}.

\begin{lemma}
  \label{lem:2}
Let\/ $\M_\Sigma$ be an irreducible trace monoid, and let $a_1\in\Sigma$. Then $p_{\Sigma}<p_{\Sigma\setminus\{a_1\}}$ and
$B_{\Sigma\setminus\{a_1\},p_\Sigma}$ is thus a well defined probability distribution on~$\M_{\Sigma\setminus\{a_1\}}$\,. 

If $\A$ is a random algorithm outputting a random element of\/ $\M_{\Sigma\setminus\{a_1\}}$
distributed according to~$B_{\Sigma\setminus\{a_1\},p_\Sigma}$\, conditionally on $\{\max(\xi) \subseteq \Lk(a_1)\}$,
then Algorithm~\ref{pojiqpsoij} described below in pseudo-code outputs with probability~$1$ an element $V\in\Pyr{a_1}$
with the same law as~$V_{a_1}$ given in~\eqref{eq:24}.
\begin{algorithm}
\caption{Outputs $V\in\Pyr{a_1}$ with the law of $V_{a_1}$}
\label{pojiqpsoij}
\begin{algorithmic}[1]
\Require{---}\Comment{No input}
\State
$\xi\gets\text{output of $\A$}$
\State
\Return $\xi\cdot a_1$
\end{algorithmic}
\end{algorithm}
\end{lemma}

\begin{proof}
The strict inequality $p_\Sigma<p_{\Sigma\setminus\{a_1\}}$ relies on the irreducibility assumption on~$\M_\Sigma$\,, and is a consequence of~\cite[Th.~4.5]{abbes17}.
%
To compute the law of the output of Algorithm~\ref{pojiqpsoij}, let $v\in\Pyr{a_1}$.
Then $v=z\cdot a_1$ for a unique $z\in\M_{\Sigma\setminus\{a_1\}}$ with $\max(z)\subseteq\Lk(a_1)$. The probability of outputting $v$ is proportional to~$p_\Sigma^{\size{z}}$, and thus to~$p_\Sigma^{\size{v}}$.

Hence, the distribution laws of the output of Algorithm~\ref{pojiqpsoij} and of the variable~$V_{a_1}$ given in~\eqref{eq:24}
have the same support, namely~$\Pyr{a_1}$, and are proportional to each other.
Consequently, they are equal to each other, which completes the proof.
%
%
\end{proof}

We are now ready to construct an algorithm achieving our goal of the uniform generation of infinite traces.

\begin{theorem}
  \label{thr:2}
Let\/ $\M_\Sigma$ be an irreducible trace monoid and let $a_1 \in \Sigma$.
Then $p_\Sigma<p_{\Sigma\setminus\{a_1\}}$\, and, if we use $p_\Sigma$ in place of~$p$,
Algorithm~\ref{algowfdoihjw2} is well defined for all inputs $(S,T)$ such that $S \subseteq \Sigma\setminus\{a_1\}$.

Moreover, the following random endless algorithm outputs at its $k^{\text{th}}$ loop an element $\xi_k\in\M_\Sigma$ with the following properties:
\begin{enumerate}
\item\label{item:5} $(\xi_k)_{k\geqslant1}$ is a non decreasing sequence.
\item\label{item:6} The element $\xi=\bigvee_{k\geqslant1}\xi_k$ is an infinite trace distributed according to the uniform measure at infinity.
\item\label{item:7} Under the same assumptions as in Theorem~\ref{thr:1}, these first $k$ loops require the execution of $\mathcal{O}(\size{\Sigma} \size{\xi_k})$ steps overall,
and the average and minimal sizes of\/ $\xi_k$ are linear in~$k$. Hence the algorithm produces in average a constant number of additional elements of\/ $\Sigma$ by unit of time.
\end{enumerate}
\begin{algorithm}
\caption{Outputs approximation of $\xi\in\BM_\Sigma$ distributed according to $\B$}
\label{aslkmalk}
\begin{algorithmic}[1]
\Require{---}\Comment{No input}
\State
$\xi\gets\unit$\Comment{Initialization}
\Repeat
\State\label{repeat-start}
 $V\gets\text{output of Algorithm~\ref{algowfdoihjw2} on input $(\Sigma\setminus\{a_1\},\Sigma\setminus\Lk(a_1))$}$ 
\State
$\xi\gets\xi\cdot V\cdot a_1$
\State\textbf{output} $\xi$\Comment{Writes on a register}\label{repeat-end}
\Until{False}
\end{algorithmic}
\end{algorithm}
\end{theorem}

\begin{proof}
As in the proof of Lemma~\ref{lem:2}, the inequality $p_\Sigma<p_{\Sigma\setminus\{a_1\}}$ relies on the irreducibility assumption on~$\M_\Sigma$\,, and follows from~\cite[Th.~4.5]{abbes17}. 

Let $(\xi_k)_{k\geqslant1}$ be the sequence of outputs of Algorithm~\ref{aslkmalk}, and let $\xi=\bigvee_{k\geqslant1}\xi_k$\,.
Then $\xi_k=(V_1\cdot a_1)\cdot\ldots\cdot(V_k\cdot a_1)$, where $V_k$ is the output of the inside \textbf{repeat} block (from line~\ref{repeat-start} to line~\ref{repeat-end}).
By construction the $(V_k\cdot a_1)_{k\geqslant1}$ random variables are \iid, and it follows from Lemma~\ref{lem:2} that their common law is that of~$V_{a_1}$\,.
Hence, the  reconstruction result for $\B$ recalled above implies that $\xi$ is distributed according to~$\B$.

It also follows from Theorem~\ref{thr:1} that every \textbf{repeat} block, when producing an element $V$,
requires executing $\mathcal{O}(\size{\Sigma} (\size{V}+1))$ steps. Hence, producing the element $\xi_k$ requires executing
$\mathcal{O}(\size{\Sigma} \size{\xi_k})$ steps overall.

Moreover, $\size{\xi_k}$ is a sum of $k$ independant \iid random variables $(\size{V_i \cdot a_1})_{1 \leqslant i \leqslant k}$,
all of which are positive integers. This completes the proof.
\end{proof}

Note that in Theorem~\ref{thr:2}, whereas the initial monoid $\M_\Sigma$ is assumed to be irreducible,
the involved submonoid $\M_{\Sigma\setminus\{a_1\}}$ may not be irreducible.
This is not an issue since all the work done in Section~\ref{sec:rand-gener-finite} for the generation of finite traces does not rely on any irreducibility assumption.

One might be concerned by the fact that the sequence $(\xi_k)_{k\geqslant1}$\,, output of Algorithm~\ref{aslkmalk}, has
a particular ``shape'', since it is the concatenation of $a_1$-pyramidal traces. For instance, if one wishes to use it for
parametric estimation or to sample some statistics on traces, the result could \emph{a priori} depend on the choice of~$a_1$\,.
But asymptotically, for a large class of statistics, the result will not depend on the choice of~$a_1$\,;
see a precise justification of this fact in~\cite{abbes16}.

\paragraph*{Generalizations.}
\label{sec:generalizations}

One may be interested in random generation techniques on trace monoids for other measures than the uniform measure. A natural generalization of the uniform measure is the notion of Bernoulli measure, developped in~\cite{abbes15}. Techniques similar to the ones presented above would apply, as it is shown in~\cite{abbes17} for trace monoids with one or zero cycle.

Another generalization might be considered. 
From a system simulation viewpoint, trace monoids lack the crucial notion of state in order to simulate ``real-life'' concurrency systems. However, this is not hopeless. The notion of trace monoid acting on a finite set provides an accurate model of such systems: see for instance how $1$-safe Petri nets fit in this model in~\cite{abbes17:_towar_petri}. A generalization of the techniques developped above is to be expected in this framework.

\section{Distributed implementation of the algorithms}
\label{sec:distr-impl-algor}

When working with concurrent systems, it is important to determine whether algorithms can be implemented in a distributed way, \emph{i.e.}, in a decentralized way.

It is known that the elements of a trace monoid have a normal form, the so-called \emph{Cartier-Foata normal form}; hence every element of $\M_\Sigma$ can be uniquely written as a product of cliques, in such a way that each factor itself is maximal, see~\cite{cartier69} (this corresponds to the greedy normal form for positive braids). The Cartier-Foata normal form can be extended to infinite traces, and then, under the uniform measure at infinity, the factors of the normal form are shown to be a discrete time Markov chain~\cite{abbes15}. Simulating this Markov chain could be an obvious way of generating infinite traces.  But it would \emph{not} be a distributed generation. Indeed, the state space of the chain is finite but huge, since it is the set of cliques of the monoid~$\M_\Sigma$\,, and the computation of the law of each factor cannot be done locally.

On the contrary, in the algorithms presented above, the only elementary operations required are concatenations of elements of the monoid.
We mention without additional details that such 
operations can be performed in a distributed way, hence without any global computation.
See \cite{abbes17} for more details on an implementation of trace monoids through tuples of words.

For some particular trace monoids, it is not necessary to precompute and store the reals $\mu_X(p)$ for all subsets $X$ of $\Sigma$,
and adequately choosing the elements $a_1$ may result in a dramatic decrease in the number of subsets $X$ for which $\mu_X(p)$ must be stored.
In these monoids, the random generation of infinite traces can also be done, as shown in~\cite{abbes17},
with bounds on the execution time similar to the bound given in Theorem~\ref{thr:2}.
These are all the trace monoids $\M(\Sigma,\R)$ such that $(\Sigma,\R)$ has zero or one cycle.

Alternatively, instead of precomputing and storing the reals $\mu_X(p)$ for all subsets $X$ of~$\Sigma$,
we may choose to recompute them on-the-fly whenever needed. In some cases, for instance if $(\Sigma,\R)$
is a graph of tree-width at most~$\kappa$, computing $\mu_X(p)$ for any subset $X$ can be performed in time~$2^\kappa \size{\Sigma}$: in these cases, omitting the precomputation phase and the need
to store the $2^{\size{\Sigma}}$ real numbers $\mu_X(p)$ is feasible, at the cost multiplying by
$2^\kappa \size{\Sigma}$ the number of steps executed in the generation phase of the algorithm.

\bibliographystyle{plain}
\bibliography{biblio}

\end{document}